\documentclass[11pt]{article}
\usepackage{graphicx} 
 
\usepackage{caption}
\usepackage{epsf,epsfig,amsfonts,amsgen,amsmath,amstext,amsbsy,amsopn,amsthm
}
\usepackage{ebezier,eepic}
\usepackage{color}
\usepackage{multirow}
\usepackage{mathrsfs}
\usepackage{tikz}
\usepackage{enumerate}
\usepackage{enumitem}
\usepackage{url}
\usepackage{pgf,tikz,pgfplots}
\usepackage{mathrsfs}
\usepackage{amssymb}
\usepackage{comment}

\usetikzlibrary{arrows}

\newtheorem{dfn}{Definition}[section]
\newtheorem{thm}[dfn]{Theorem}
\newtheorem{lem}[dfn]{Lemma}

\newtheorem{corollary}[dfn]{Corollary}
\newtheorem{conjecture}[dfn]{Conjecture}

\newtheorem{Claim}[dfn]{Claim}

 \newcommand{\dF}{{\mathcal{F}}}
  \newcommand{\dG}{{\mathcal{G}}}  
  \newcommand{\dH}{\mathcal{H}}
\newcommand{\dK}{\mathcal{K}}
\newcommand{\dP}{{\mathcal{P}}}
\newcommand{\dS}{{\mathcal{S}}}

\title{Maxmum Size of a Uniform Family with Bounded VC-dimension}
\author{Tianchi Yang\footnote{School of Mathematics, Georgia Institute of Technology, Atlanta, GA, USA}~~~~ Xingxing Yu \footnote{School of Mathematics, Georgia Institute of Technology, Atlanta, GA, USA. Partially supported by NSF grant DMS-2348702.}}
\date{}

\begin{document}

\maketitle

\begin{abstract} 
In 1984, Frankl and Pach \cite{84FP} proved that, for positive integers $n$ and $d$, the maximum size of a $(d+1)$-uniform set family $\mathcal{F}$ on an $n$-element set with  VC-dimension at most $d$ is at most ${n\choose d}$; and they suspected that ${n\choose d}$ could be replaced by ${n-1\choose d}$, which would generalize the famous Erd\H{o}s-Ko-Rado theorem and was mentioned  by Erd\H{o}s \cite{84Er} as Frankl--Pach conjecture.  
However, Ahlswede and Khachatrian \cite{97AK} in 1997 constructed $(d+1)$-uniform families on an $n$-element set with VC-dimension at most $d$ and size exactly $\binom{n-1}{d}+\binom{n-4}{d-2}$, and Mubayi and Zhao \cite{07MZ} in 2007 constructed more such families.  It has since been an open question to narrow the gap  between the lower bound $\binom{n-1}{d}+\binom{n-4}{d-2}$ and the upper bound ${n\choose d}$. In a recent breakthrough, Chao, Xu, Yip, and Zhang \cite{CXYZ} reduced the upper bound $\binom{n }{d}$ to $ \binom{n-1}{d}+O( n^{d-1-\frac{1}{4d-2}})$. In this paper, we further reduce the upper bound to $\binom{n-1}{d} + O(n^{d-2})$, 
asymptotically matching the lower bound $\binom{n-1}{d}+\binom{n-4}{d-2}$. 
\end{abstract}


\section{Introduction}
 
A fundamental result in extremal set theory is the Erd\H{o}s--Ko--Rado theorem \cite{61EKR}: For positive integers $n,k$ with $n\ge 2k$, ${n-1\choose k-1}$ is the maximum size of an intersecting family of $k$-subsets of $[n]$, and when $n>2k$ all maximum intersecting families are trivially intersecting. An attempt to generalize this result was initiated by Frankl and Pach \cite{84FP} and mentioned by Erd\H{o}s \cite{84Er}, by considering the maximum size of a uniform set family on a given set and with bounded VC-dimension.

Given a family of sets $\dF\subseteq 2^{[n]}$, the {\em VC-dimension} of $\dF$, denoted by VC$(\dF)$, is the largest integer $d$ for which there exists a set $S\in  {[n]\choose d}$ that is shattered by $\dF$. A set $S$ is {\em shattered} by $\dF$ if, for every subset $A\subseteq S$, there exists some $F\in \dF$ such that $F\cap S=A$.

The concept of VC-dimension was introduced and studied independently and  simultaneously in the early 1970s by Perles–Shelah  \cite{72Sh}, Sauer \cite{72Sa}, and
Vapnik and Chervonenki \cite{71VC}. There has been recent progress on several difficult problems in combinatorics for set families with bounded VC-dimension, such as the Erd\H{os}-Hajnal conjecture \cite{19FPS, 23NSS}, the Schur-Erd\H{o}s conjecture \cite{21FPS}, and bounds on sunflowers in set systems \cite{23FPS, 24BBDFP}. 
VC-dimension is also a fundamental concept in learning theory, 
see for example \cite{15MSWY, 19BH}, 
that provides a formal measure of complexity/capacity of a machine learning model.  (VC dimension helps to quantify how flexible a model is in its ability to fit diverse patterns in data and establish theoretical bounds on the amount of training data needed to ensure that the model can generalize effectively.)

Let $\dF \subseteq \binom{[n]}{d+1}$ be a $(d+1)$-uniform family. 
Frankl and Pach \cite{84FP} showed that if  VC$(\dF)=k\le d$ then $|\dF|\le \binom{n}{k}$. 
They felt that the following could be true: For sufficiently large $n$, if  VC$(\dF)\le d$ then $|\dF|\le \binom{n-1}{d}$. Erd\H{o}s remarked \cite{84Er} (also see \cite{84FP, 94FW}) that this Frankl-Pach conjecture would be a beautiful generalization of the Erd\H{o}s-Ko-Rado theorem.
However, Ahlswede and Khachatrian \cite{97AK} constructed, for $n\ge 2(d+1)$, a $(d+1)$-uniform  family of size $\binom{n-1}{d}+\binom{n-4}{d-2}$ with VC dimension $d$. 
Mubayi and Zhao \cite{07MZ} provided more such families, and believed that $\binom{n-1}{d}+\binom{n-4}{d-2}$ is the right upper bound.

For nearly four decades, the Frankl-Pach bound ${n\choose d}$ remained as the best upper bound on the size of a subfamily of ${[n]\choose d+1}$ with VC-dimension at most $d$. Mubayi and Zhao \cite{07MZ} improved it by an additive factor of $\log n$ for prime power $d$ and large $n$, using linearly independent vectors and a sunflower lemma of Erd\H{o}s and Rado~\cite{60ER} (see Lemma~\ref{lem: sunflower} below).  In 2024,  Ge, Xu, Yip, Zhang, and Zhao \cite{24GXYZZ} improved the Frankl-Pach bound from $\binom{n}{d}$ to $\binom{n}{d}-1$, for all $d\ge 2$ and $n\ge 2d+2$, by using multilinear polynomials instead of vectors. 
In a recent breakthrough,  Chao, Xu, Yip, and Zhang \cite{CXYZ} used a combinatorial approach to show that $|\dF|\le \binom{n-1}{d}+O(n^{d-1-\frac{1}{4d-2}}).$ (We use the standard little $o$ and big $O$ notation.) 
 In this paper, we prove the following result, which provides an upper bound that asymptotically matches the lower bound.

\begin{thm}\label{thm: main}
    Let $n,d$ be positive integers, and let $\dF\subseteq \binom{[n]}{d+1}$ be a maximum size family with VC$(\dF) \le d$. Then
    $$
    \binom{n-1}{d}+\binom{n-4}{d-2} \le |\dF|\le \binom{n-1}{d} + O(n^{d-2}).
    $$
\end{thm}

This result gives a substantial improvement over the previous upper bounds, as the error term in the upper bound now matches the order of magnitude of the second term in the lower bound.  To describe our proof strategy for Theorem~\ref{thm: main}, we need some notation. 
 For a set family $\dF$, we will use $\partial F$ to denote the {\it shadow} of $\dF$, i.e., 
 \[\partial \dF:=\{S: S\subseteq F \mbox { and } |S|=|F|-1 \mbox{ for some } F\in \dF\}.\]

Let $n,d$ be positive integers and let $\dF\subseteq \binom{[n]}{d+1}$ with VC$(\dF) \le d$. Then, for any $\dG\subseteq \dF$, VC$(\dG)\le$ VC$(\dF)$, since if $\dG$ shatters a set $F$ then so does $\dF$. Moreover,  for each $F\in\dF$, $F$ is not shattered by $\dF$; so    there exists a subset $T\subseteq F$ such that  $F'\cap F\neq T$ for all $F'\in \dF$. 
We call such $T$ an \emph{$\dF$-certificate} of $F$, which is a proper subset of $F$ and  need not be unique. A {\em maximum} $\dF$-certificate of $F$ is an  $\dF$-certificate of $F$ with the maximum size. Let $c_{\dF}:\dF \to \binom{[n]}{\le d}$ denote any  function such that, for each $F\in \dF$, $c_{\dF}(F)$ is a maximum $\dF$-certificate of $F$. 
Then, for every $T\in {[n]\choose \le d}$, $c_{\dF}^{-1}(T)\subseteq \{F\in \dF:T\subseteq F\}$. 
For $k\in [d]$, let $\dF_k:=\{F\in \dF: |c_{\dF}(F)|=k\}$. Note that $\dF_k$ is independent of the choice of the assignment function $c_{\dF}$.

In our proof of Theorem \ref{thm: main}, we will partition $\dF$ into three subfamilies $\dF^i$, $i\in [3]$, such that
\begin{itemize}
    \item $\dF^1=\dF_{\le d-2}\cup \dP$ and $|\dF^1|\le 0.1|\overline {\partial \dF}|+O(n^{d-2})$, where $\dF_{\le d-2}:=\bigcup_{k=1}^{d-2}\dF_k$, $\dP$ is some subfamily of $\dF_{d-1}$, and $\overline{\partial \dF}={[n]\choose d}\setminus \partial \dF$;     
    \item $\dF^2$ consists of sets $F$ in $\dG:=\dF\setminus \dF^1$ satisfying certain constraints on $c_{\dG}(F)$, and $|\dF^2|=O(n^{d-2})$;
    \item there is an injection from $\dF^3:=\dF\setminus (\dF^1\cup \dF^2)$ to the standard basis of the vector space $ \mathbb{R}^m$, for some appropriate $m\le {\binom{n-1}{d}}$. 
\end{itemize}

In Section 2, we explain how $\dP$ (and hence $\dF^1$) is chosen such that $|\dF^1|\le 0.1|\overline{\partial \dF}|+ O(n^{d-2})$ and  $\dG:=\dF\setminus \dF^1$ has some desired properties, see Lemma~\ref{lem: good subfamily}. We then prove Lemma~\ref{lem: final refinement} which will be used to identify the subfamily $\dF^2\subseteq \dG$.
In Section 3, we prove Lemma~\ref{lem: f G'} and define a function $f$ from $\dF^3$ to the standard basis of $ \mathbb{R}^m$ for some $m\le \binom{n-1}{d}$. In Section 4, we modify $f$ to obtain the desired injection, see Lemma~\ref{lem: map G' to 1}.  We then complete the proof of Theorem~\ref{thm: main} and  mention several related open problems.



\section{Construction of $\dF^1$ and  $\dF^2$}

Let $n,d$ be positive integers, let $\dF\subseteq \binom{[n]}{d+1}$ with VC$(\dF) \le d$, and let $c_{\dF}$ be an assignment of maximum $\dF$-certificates. In this section, we describe the family $\dF^1$ in the partition of $\dF=\dF^1\cup \dF^2\cup \dF^3$ mentioned in Section 1.  To do this we need to consider   $|c_{\dF}^{-1}(T)|$ for $T\in  {[n]\choose \le d}$, which is the number of sets in $\mathcal{F}$ with $T$ as a maximum $\dF$-certificate under $c_{\dF}$. Using a bound from \cite{24BBDFP} on sizes of sunflowers in set families with bounded VC-dimension, Chao {\em et al}  proved that $|c_{\dF}^{-1}(T)|$ is bounded from above by a constant dependent only on $d$, which is Claim 3.2 in \cite{CXYZ} and stated below. 

\begin{lem}  [Chao, Xu, Yip, and Zhang]\label{clm: T-F-number}
Let $n,d$ be positive integers, let $\dF\subseteq \binom{[n]}{d+1}$ with VC$(\dF) \le d$, and let $c_{\dF}$ be an assignment of maximum $\dF$-certificates of the sets in $\dF$.  Then, for all $T \in {[n]\choose \le d}$,  $|c_{\dF}^{-1}(T)|=O(1)$.
\end{lem}
For the sake of completeness, we restate their proof which uses Lemma~\ref{lem: sunflower}, the famous sunflower lemma of Erd\H{o}s and Rado \cite{60ER}. (The bound in Lemma~\ref{lem: sunflower} was improved by Alweiss, Lovett, Wu, and Zhang \cite{21ALWZ} and then by Bell, Chueluecha, and Warnke \cite{21BCW}.) For a positive integer $p$, a {\it $p$-sunflower} is a family of $p$ sets $G_1,\ldots, G_p$ with the property that there exists some set $S$ such that $G_i\cap G_j=S$ for $i,j\in [p]$ with $i\ne j$. 

\begin{lem}[Erd\H{o}s-Rado Sunflower Lemma]\label{lem: sunflower}
Let $k,p$ be positive integers. If $\dF$ is a family of $k$-sets of size $|\dF|\ge k!(p-1)^k$, then $\dF$ contains a $p$-sunflower.
\end{lem}

\begin{proof}[Proof of Lemma~\ref{clm: T-F-number}]
    Assume $|c_{\dF}^{-1}(T)|\ge (d+1)!(d+2-1)^{d+1}$ (which we take as the $O(1)$ term in Lemma~\ref{clm: T-F-number}).  Then by  Lemma~\ref{lem: sunflower}, $c_{\dF}^{-1}(T)$ contains a $(d+2)$-sunflower, say $\{G_1,G_2,\cdots,G_{d+2}\}$.
    Let $A=\cap_{i\in[d+2]}G_i$; so $T$ is a proper subset of $A$. For each $i\in[d+2]$, let   $B_i=G_i\setminus A$. By the definition of sunflower, $B_1, \ldots, B_{d+2}$ are pairwise disjoint. 
    Consider the set $G_1=A\cup B_1$. Since $T\cup B_1$ is not a certificate of $G_1$, there exists an $F\in \dF$ with $F\cap G_1=T\cup B_1$. It follows that $F\cap A=T$. Since $B_1, \ldots, B_{d+2}$ are pairwise disjoint, we can find an index $j\in [d+2]$ such that  $B_j\cap F=\emptyset$. Therefore, $F\cap G_j=F\cap A=T=c_{\dF}(G_j)$, which is a contradiction.
\end{proof}

Recall the notation $\dF_k=\{F\in \dF:   |c_{\dF}(F)| = k \}$ and $\dF_{\le d-2} = \bigcup_{k \le d-2} \dF_k$. We now use
Lemma~\ref{clm: T-F-number} to bound the size of $\dF_{\le d-2}$. 

\begin{lem}\label{clm: F2 size}
   Let $n,d$ be positive integers and let $\dF\subseteq \binom{[n]}{d+1}$ with VC$(\dF) \le d$. Then  $|\dF_{\le d-2}| = O(n^{d-2})$.
\end{lem}
\begin{proof}
   Let $c_{\dF}$ be an assignment of maximum $\dF$-certificates of the sets in $\dF$. By Lemma~\ref{clm: T-F-number}, $|c_{\dF}^{-1}(T)|=O(1)$ for all $T\in {[n]\choose \le d}$. Therefore,
    \[|\dF_{\le d-2}| \le \sum_{k=1}^{d-2} \sum_{T\in \binom{[n]}{k}} O(1)= \sum_{k=1}^{d-2} \binom{n}{k} \cdot O(1) = O(n^{d-2}).\]
\end{proof}

We need more information on sets in $\dF_d\cup \dF_{d-1}$. 
For convenience, for $T\subseteq [n]$ and distinct elements $x_1,\ldots,x_r\in [n]\setminus T$, we write $Tx_1\ldots x_r$ for $T\cup \{x_1,\ldots, x_r\}$.

\begin{lem} \label{clm: aT details}
Let $n,d$ be positive integers, let $\dF\subseteq \binom{[n]}{d+1}$ with VC$(\dF) \le d$, and let $c_{\dF}$ be an assignment of maximum $\dF$-certificates of the sets in $\dF$. Let $T \in c_{\dF}(\dF)$. 
      If $|T|=d$ then $|\{F\in \dF: T\subseteq F\}|=1$ (thus $|c_{\dF}^{-1}(T)|=1$).   
      If $|T|=d-1$ then  one of the following holds:
      
            \begin{itemize}
                \item [(i)] There exist distinct elements $x,y,z\in [n]\setminus T$ such that $c_{\dF}^{-1}(T)=\{F\in \dF: T\subseteq F\}=\{Txy,Tyz,Tzx\}$. 
               
                \item [(ii)]  There exist distinct elements $x,y,z\in [n]\setminus T$ such that $c_{\dF}^{-1}(T)=\{Txy,Txz\}$ and there exists $U\subseteq [n]\setminus (T\cup \{x,y,z\})$ such that  $\{F\in \dF: T\subseteq F\}=\{Txy,Txz\}\cup \{Txu: u\in U\}$    or $\{F\in \dF: T\subseteq F\}=\{Txy,Txz,Tyz\}\cup \{Txu: u\in U\}$.  
            
                \item [(iii)] There exist distinct elements $x,y\in [n]\setminus T$ such that $c_{\dF}^{-1}(T)=\{Txy\}$ and  there  exist subsets $U,V\subseteq [n]\setminus (T\cup \{x,y\})$ such that  $\{F\in \dF: T\subseteq F\}=\{Txy\}\cup \{Txu:u\in U\}\cup \{Tyv:v\in V\}$.
            \end{itemize}
 
\end{lem}
\begin{proof}
    First, suppose $|T|=d$. Let $x\in [n]\setminus T$ such that $c_{\dF}(Tx)=T$. Then, for any $F\in \dF$, $F\cap Tx\ne T$; so $\{F\in \dF: T\subseteq F\}=\{Tx\}$. Therefore, $|\{F\in \dF: T\subseteq F\}|=1$.

    Now consider the case $|T|=d-1$. Let $x,y\in [n]\setminus T$ such that $c_{\dF}(Txy)=T$. Then for any two distinct elements $u,v\in [n]\setminus T$ such that $Tuv\in \dF$, we have $Txy\cap Tuv\neq T$; so $\{u,v\}\cap \{x,y\}\neq \emptyset$. 
    
    If $c_{\dF}^{-1}(T)=\{Txy\}$ then,  for any $F\in \dF$ such that $T\subseteq F$ and $F\ne Txy$, we have $F=Txu$ for some $u\in [n]\setminus (T\cup \{x,y\})$ or $F=Tyv$  for some $v\in [n]\setminus (T\cup \{x,y\})$; so (iii) holds. Hence, we may assume that  there exists $z\in [n]\setminus (T\cup \{x,y\})$ such that $Txz\in c_{\dF}^{-1}(T)$ or $Tyz\in c_{\dF}^{-1}(T)$. Without loss of generality, we may assume $Txz\in c_{\dF}^{-1}(T)$. 
    
    Suppose we also have $Tyz\in c_{\dF}^{-1}(T)$. Then, for any distinct elements $u,v\in [n]\setminus T$ such that $Tuv\in \dF$, we must have $\{u,v\}\cap \{x,y\}\ne \emptyset$, $\{u,v\}\cap \{y,z\}\ne \emptyset$, and $\{u,v\}\cap \{z,x\}\ne \emptyset$. This implies $\{u,v\}\subseteq \{x,y,z\}$. Hence, (i) holds. 

    Now assume $Tyz\notin c_{\dF}^{-1}(T)$. Then for any distinct $u,v\in [n]\setminus T$ such that $Tuv\in \dF$, $\{u,v\}\cap \{x,y\}\ne \emptyset$ and $\{u,v\}\cap \{x,z\}\ne \emptyset$. Thus, $\{u,v\}= \{y,z\}$ or $x\in \{u,v\}$. So if $F\in \dF$ and $T\subseteq F$ then $F=Tyz$ or $F=Txu$ for some $u\in [n]\setminus (T\cup \{x,y,z\})$. Hence, we have (ii). 
    \end{proof}

To define $\dF^1$ for the desired partition of $\dF$, we will include all sets in $\dF_{\le d-2}$, as well as the sets in a family $\dP\subseteq \dF_{d-1}$ as described in the lemma below.

\begin{lem}\label{clm: P size}
Let $n,d$ be positive integers, let $\dF\subseteq \binom{[n]}{d+1}$ with VC$(\dF) \le d$, and let $c_{\dF}$ be an assignment of maximum $\dF$-certificates of the sets in $\dF$. Let $\dP:=\{F_i,F_i': i\in [k]\}$ denote a  collection of pairwise distinct sets, such that, for $i\in [k]$, $F_i,F_i'\in \dF_{d-1}$ and $c_{\dF}(F_i)\cup c_{\dF}(F_i')=F_i\cap F_i'$. 
Then
    $|\dP|\le  0.1|\overline{\partial \dF}| +O(n^{d-2})$, where $\overline{\partial\dF}={[n]\choose d}\setminus \partial \dF$.
\end{lem}

\begin{proof}
  By definition, we have $|c_{\dF}(F_i)|=|c_{\dF}(F_i')|=d-1$ (as $F_i,F_i'\in \dF_{d-1}$) and $|F_i|=|F_i'|=d+1$. Since $c_{\dF}(F_i) \cup c_{\dF}(F_i') = F_i\cap F_i'$ and $c_{\dF}(F_i) \neq F_i\cap F_i'$, we must have $|F_i\cap F_i'|=d$ and $|c_{\dF}(F_i)\cap c_{\dF}(F_i')|=d-2$.  Hence, for each $i\in [k]$, we may let $T_i:=c_{\dF}(F_i)\cap c_{\dF}(F_i')$ and $x_i,y_i\in [n]\setminus T_i$, such that $c_{\dF}(F_i)= Tx_i, c_{\dF}(F_i')= Ty_i$, and  $ F_i \cap   F_i' = Tx_iy_i $. Let $p(T):=|\{i\in [k]: T_i=T\}|$. Then 
    $$  \sum_{T\in \binom{[n]}{d-2}} p(T)=|\dP|/2.$$
    
   We claim that  if $i,j\in [k]$ such that $T_i=T_j=T$ and $x_i,y_i,x_j,y_j$ are all distinct then either $Tx_ix_j,Ty_iy_j\in \overline{\partial \dF}$ or  $Tx_iy_j,Ty_ix_j\in \overline{\partial \dF}$. To see this, let $F_i=Tx_iy_ia_i$, $F_i'=Tx_iy_ib_i$, $F_j=Tx_jy_ja_j$, and $F_j'=T x_jy_jb_j$, where $a_i\neq b_i$ and $a_j\neq b_j$. Note that $a_i=a_j$ or $b_i=b_j$ implies $a_i\ne b_j$ and $b_i\ne a_j$; so we have $a_i\neq a_j$ and $b_i\neq b_j$, or  $a_i\neq b_j$ and  $b_i\neq a_j$. By symmetry, we may assume $a_i\neq a_j$ and $b_i\neq b_j$. If $Tx_ix_j\notin \overline{\partial\dF}$ then by definition there exists $z\in [n]\setminus Tx_ix_j$ such that $Tx_ix_jz\in \dF$; now either $Tx_ix_jz\cap Tx_iy_ia_i=Tx_i=c_{\dF}(F_i)$ or $Tx_ix_jz\cap Tx_jy_ja_j=Tx_j=c_{\dF}(F_j)$,  contradicting the definition of $\dF$-certificates. Hence, $Tx_ix_j\in \overline{\partial \dF}$. Similarly, we have $Ty_iy_j\in \overline{\partial \dF}$ (using $b_i,b_j, y_i,y_j$ in place of $a_i,a_j,x_i,x_j$, respectively), completing the proof of this claim.

    We now count $M$,  the number of quadruples $(i,j,T,D)$ satisfying the following properties:
    \begin{itemize}
    \item  $i,j\in [k]$ and $i\ne j$, 
    \item $T_i=T_j=T$, and $x_i,y_i,x_j,y_j$ are all distinct, and 
    \item  $D\in \overline{\partial \dF}\cap \{Tx_ix_j,Ty_iy_j,Tx_iy_j,Ty_ix_j\}$.
    \end{itemize}
    
     First, fix $T\in \binom{[n]}{d-2}$ and let $m_1,m_2,\cdots ,m_{p(T)}\in [k]$ such that  $T_{m_r}=T$ for all $r\in [p(T)]$. Thus, for each $m_r$, $c_{\dF}(F_{m_r})=Tx_{m_r}$ and $c_{\dF}(F_{m_r}')=Ty_{m_r}$.  
    Since $|c_{\dF}^{-1}(Tx_{m_r})|\le 3$ and $|c_{\dF}^{-1}(Ty_{m_r})|\le 3$ (by Lemma \ref{clm: aT details}), there are at least $p(T)-6$ indices $m_s$ such that $x_{m_s},y_{m_s}, x_{m_r},y_{m_r}$ are distinct. By the above claim, each  such pair \( \{m_r, m_s\}\)  gives rise to at least two desired quadruples.
    So in total, there are at least $(p(T)-6)p(T)/2$ such pairs $\{m_r,m_s\}$ with  $r,s\in [p(T)]$,  contributing at least $(p(T)-6)p(T)$ desired quadruples $(i,j,T,D)$. Therefore
    $$M\ge \sum_{T\in \binom{[n]}{d-2}}(p(T)-6)p(T) 
    \ge \sum_{T\in \binom{[n]}{d-2}} p(T) ^2-3|\dP| \ge \frac{ |\dP|^2}{ 4\binom{n}{d-2}} -3|\dP|,$$ 
    where the final inequality is an application of the Cauchy-Schwarz inequality.  
      
    On the other hand,  fix a set $D\in \overline{\partial \dF}$ and let $T\in \binom{D}{d-2}$, such that $(i,j,T,D)$ is a quadruple counted by $M$. Let $D\setminus T=\{u,v\}$. Since $D\in \overline{\partial \dF}\cap \{Tx_ix_j,Ty_iy_j,Tx_iy_j,Ty_ix_j\}$, we may assume that 
    $u\in \{x_i,y_i\}$ and $v\in \{x_j,y_j\}$.
    Then $Tu=c_{\dF}(F_i)$ or $Tu=c_{\dF}(F_i')$. Since $|c_{\dF}^{-1}(Tu)|\le 3$ (by Lemma~\ref{clm: aT details}),  there are at most 3 possible choices for \( i \). Similarly, by considering $Tv$ (and $F_j,F_j'$),  there are at most 3 possible choices for \( j \).   Note that there are  $| \overline{\partial \dF}|$ choices for $D$ and, for each choice of $D$, there are at most $\binom{d}{d-2}$ choices for  $T$. Hence, the total number of quadruples is 
    $$M\le 3\times 3\times | \overline{\partial \dF}|\times \binom{d}{d-2}<9d^2 | \overline{\partial \dF}|.$$ 
    
    Combining the above bounds on \( M \), we obtain
    \[9d^2 | \overline{\partial \dF}|>    \frac{|\dP| ^2}{4\binom{n}{d-2}} -3|\dP| .\]
    When $|\dP|>400d^2n^{d-2} $, the above inequality implies
    $9d^2 | \overline{\partial \dF}|\ge    97 d^2 |\dP| $; so  $  0.1| \overline{\partial \dF}|\ge      |\dP| $. 
    If, instead, $|\dP|\le 400d^2n^{d-2} $ then $| \dP |=O(n^{d-2})$. 
    Therefore,  $|\dP| \le 0.1|\overline{\partial \dF}| +O(n^{d-2}).$
\end{proof}

For the desired partition of $\dF$, we will set  $\dF^1=\dF_{\le d-2}\cup \dP$ (as described in Lemma~\ref{clm: P size}). 
The next lemma (and its proof) describes useful properties of the family $\dG:=\dF\setminus \dF^1$ which will enable us to partition $\dG$ to $\dF^2$ and $\dF^3$. Recall the notation $\overline{\partial \dG}={[n]\choose d}\setminus \partial \dG$, where $\partial(\dG)$ is the shadow of $\dG$.

\begin{lem}\label{lem: good subfamily}
    Let $n,d$ be positive integers, and let $\dF\subseteq \binom{[n]}{d+1}$ be a family with VC$(\dF) \le d$. There exist a subfamily $\dG\subseteq \dF$ and an assignment $c_{\dG}: \dG\to {[n]\choose \le d}$ of maximum $\dG$-certificates of the sets in $\dG$,    such that  
    \begin{itemize}
        \item [(i)] $|\dF\setminus \dG|\le 0.1|\overline{\partial \dF}|+O(n^{d-2})$ and $|\dG| \ge |\dF|  - 0.1 |\overline{\partial \dG}|- O(n^{d-2})$, 
        \item[(ii)]  $d-1\le |c_{\dG}(F)|\le d$ for every $F\in \dG$, and
        \item[(iii)] if $F,F'\in \dG$ such that $|c_{\dG}(F)|=|c_{\dG}(F')|=d-1$ and $c_{\dG}(F)\ne c_{\dG}(F')$, then the families $S_{F,c_\dG}:=\{S\in {F\choose  d-1}\cup {F\choose  d}: c_{\dG}(F)\subseteq S\}$ and $S_{F',c_{\dG}}:=\{S\in {F'\choose  d-1}\cup {F'\choose  d}: c_{\dG}(F')\subseteq S\}$ are disjoint. 
    \end{itemize}
\end{lem}

\begin{proof}
   Let $c_{\dF}$ be an assignment of maximum $\dF$-certificates of the sets in $\dF$. Let $\dP:=\{F_i,F_i': i\in [k]\}$ denote a  collection of pairwise distinct sets, such that, for $i\in [k]$, $F_i,F_i'\in \dF_{d-1}$ and $c_{\dF}(F_i)\cup c_{\dF}(F_i')=F_i\cap F_i'$. Choose $\dP$ to be a maximal such collection,   and let $\dG:=(\dF_d\cup \dF_{d-1})\setminus \dP$. By Lemma
   ~\ref{clm: P size}, $|\dP|\le 0.1|\overline{\partial \dF}|+O(n^{d-2})$; and by Lemma~\ref{clm: F2 size}, $|\dF_{\le d-2}|=O(n^{d-2})$. Hence, $$|\dF\setminus \dG|=|\dF_{\le d-2}\cup \dP|\le 0.1|\overline{\partial \dF}|+O(n^{d-2}).$$ Moreover,   
    \[|\dG|=|\dF|-| \dF_{\le d-2}| - |\dP|\ge  |\dF|- O(n^{d-2}) -  (0.1|\overline{\partial \dF}| + O(n^{d-2}))
    \ge |\dF|-  0.1|\overline{\partial \dG}|- O(n^{d-2}), \]
    where the last inequality holds since  $\overline{\partial \dF}\subseteq \overline{\partial \dG}$ (as $\dG\subseteq \dF$). So (i) holds.

    Next, we define $c_{\dG}: \dG \to {[n]\choose \le d}$ such that, for each $F\in \dG$, $c_{\dG}(F)$ is a maximum $\dG$-certificate of $F$. Note that for each $F\in \dG$, $c_{\dF}(F)$ is also a $\dG$-certificate of $F$, but need not be  a maximum one. Let $F\in \dG$. If the size of a maximum $\dG$-certificate of $F$ is  $d-1$ then $|c_{\dF}(F)|=d-1$ and $c_{\dF}(F)$ is also a maximum $\dG$-certificate of $F$; so in this case, define $c_{\dG}(F)=c_{\dF}(F)$.   If the size of a maximum $\dG$-certificate of $F$ is $d$ then define $c_{\dG}(F)$ to be any maximum $\dG$-certificate of $F$. Clearly, (ii) holds.

    To prove (iii), let $T, T'\in c_{\dG}(\dG)\cap {[n]\choose d-1}$ be distinct, and let $F\in c_{\dG}^{-1}(T)$ and $F'\in c_{\dG}^{-1}(T')$. Suppose, for a contradiction,  there exists $S\in {[n]\choose  d-1}\cup {[n]\choose  d}$ such that $T\subseteq S\subseteq F$ and $T'\subseteq S\subseteq F'$. By considering the sizes of these sets involved, we have $T\cup T'=S= F\cap F'$. Since $|T|=|T'|=d-1$, we know from the above definition of $c_{\dG}$  that $c_{\dF}(F) =c_{\dG}(F) =T$ and $ c_{\dF}(F')=c_{\dG}(F')=T'$. Therefore $c_{\dF}(F)\cup c_{\dF}(F') =F\cap F'$ and, hence, we may add the pair $\{F,F'\}$ to $\dP$ to form a larger collection, contradicting the maximality of $\dP$.    
\end{proof}

   To further partition the family $\dG$ from Lemma~\ref{lem: good subfamily} to obtain the desired partition $\dF=\dF^1\cup \dF^2\cup \dF^3$, we will set $\dF^2:=\dG(i,j)\cup  \dG_{d-1}(i)\cup \dG_{d-1}(j)$ which are described in (ii) of the following lemma.
   
  \begin{lem}\label{lem: final refinement}
  Let $n,d$ be positive integers and let $\dF\subseteq \binom{[n]}{d+1}$ be a family with VC$(\dF) \le d$. Let $\dG\subseteq \dF$ and let $c_{\dG}$ be an assignment of $\dG$-certificates of the sets in $\dG$, such that $c_{\dG}(\dG)\subseteq {[n]\choose d-1}\cup {[n]\choose d}$.  Then there exist distinct elements $i,j\in [n]$, such that 
  \begin{itemize}
      \item [(i)] $|{[n]\setminus \{i,j\}\choose d}\cap \overline{\partial \dG}|= (1-o(1))|\overline{\partial \dG}|$, and 
      \item [(ii)] $\dG_{d-1}(i):=\{F\in \dG_{d-1}: i\in F \}$,  $\dG_{d-1}(j):=\{F\in \dG_{d-1}: j\in F\}$, and  $\dG(i,j):=\{F\in \dG: \{i,j\}\subseteq F \mbox{ and } \
      |c_{\dG}(F)\setminus \{i,j\}|\le d-2\}$ 
      all have size $O(n^{d-2})$.   
  \end{itemize}    
  \end{lem}
 \begin{proof} 
For $r\in [n]$, let $\dG_{d-1}(r)=\{F\in \dG_{d-1}: r\in F\}$, i.e., $\dG_{d-1}(r)=\{F\in \dG: r\in F \mbox{ and } |c_{\dG}(F)|=d-1\}$.  For each set $T\in \binom{[n]}{d-1}$, we have $|c_{\dG}^{-1}(T)|\le 3$ by Lemma~\ref{clm: aT details}. Hence, 
$ \sum_{r\in[n]} |\dG_{d-1}(r)|\le 3(d+1)\binom{n}{d-1}. $
By averaging, we see that there are at least $n/2+1$ choices of $r$ satisfying $|\dG_{d-1}(r)| = O(n^{d-2})$. 
Likewise, for $r\in [n]$, let $\overline{\partial \dG}(r):=\{F\in \overline{\partial \dG}: r\in F\}$; then $\sum_{r\in [n]}\overline{\partial \dG}(r)\le d|\overline{\partial \dG}|$ and, hence, at least $n/2+1$ choices of $r$ satisfy $|\overline{\partial \dG}(r)|=o(|\overline{\partial \dG}|)$. Thus, by the pigeonhole principle, there exist distinct $i,j\in [n]$, such that \begin{itemize}
    \item [(1)] $|\overline{\partial\dG}(i)|=o(|\overline{\partial \dG}|)$ and $|\overline{\partial\dG}(j)|  =o(|\overline{\partial \dG}|)$, and  
    \item [(2)] $|\dG_{d-1}(i)|=O(n^{d-2})$ and $|\dG_{d-1}(j)|= O(n^{d-2})$. 
    \end{itemize}
    
Now (1) implies 
$\left|\binom{[n]\setminus \{i,j\}}{d}\cap \overline{\partial \dG} \right|= (1-o(1))|\overline{\partial \dG} |$;  so (i) holds. 
In view of (2), it remains to show that $|\dG(i,j)|=O(n^{d-2})$.
 Note that if $T=c_{\dG}(F) \setminus \{i,j\}$ for some $F\in \dG(i,j)$ then, since  $d-1\le |c_{\dG}(F)|\le d$ (by assumption),
 we have $d-3\le |T|\le d-2$.   
 Thus, the number of $T$ satisfying $T=c_{\dG}(F) \setminus \{i,j\}$ for some $F\in \dG(i,j)$ is  at most $\binom{n-2}{d-2}+\binom{n-2}{d-3}=\binom{n-1}{d-2}$, and $c_{\dG}(F)\in \{T\cup \{i\}, T\cup \{j\},T\cup \{i,j\} \}$. 
 Therefore, $|c_{\dG}(\dG(i,j))|\le \binom{n-1}{d-2} \times 3$. By Lemma~\ref{clm: aT details}, at most 3 sets in $\dG(i,j)$ share the same set in $c_{\dG}(\dG(i,j))$ as their maximum $\dG$-certificate. 
 Hence,
$|\dG(i,j)|\le 3|c_{\dG}(\dG(i,j))|\le  3 \times \binom{n-1}{d-2}  \times 3=O(n^{d-2}) $. So (ii) also holds. 
\end{proof}




\section{Mapping $\dF^3$ to a vector space}

In this section we use Lemmas~\ref{lem: good subfamily} and \ref{lem: final refinement} to complete the partition of $\dF$ to $\dF^1\cup \dF^2\cup \dF^3$ and define a function $f$ from $\dF^3$ to $\mathbb{R}^m$ (where $m$ is the size of some set family) which will be modified to an injection later in Section 4. For a set family $\dS$, we let $\mathbf U_{\dS}:=\{\mathbf u_S:S\in \dS\}$ be the set of standard basis vectors of $\mathbb{R}^{|\dS|}$ indexed by a fixed linear ordering of the sets in $\dS$. Thus, for any two sets $S,T\in \dS$, the inner product  
\[
\mathbf u_S\cdot \mathbf u_T=
\begin{cases}
1 \quad \text{ if } T=S;\\ 
0 \quad \text{ if } T\ne S. 
\end{cases}
\]
We can now state and prove the main lemma of this section.

\begin{lem}\label{lem: f G'}
Let $n,d$ be positive integers and let $\dF\subseteq \binom{[n]}{d+1}$ be a family with VC$(\dF) \le d$. Then there is a partition of $\dF$ into families $\dF^1$, $\dF^2$ and $\dF^3$,  and there exist $V\in {[n]\choose n-2}$ and a function $f: \dF^3 \to \mathbf U_{\dS}$, where  $\dS=\binom{V}{d-1} \cup \left(\binom{V}{d}\setminus \overline{\partial \dF^3}\right)$ and $\overline{\partial \dF^3}={[n]\choose d}\setminus \partial \dF^3$,
such that  

\begin{enumerate}
      \item [(i)] $|\dF^1\cup \dF^2|\le 0.1|\overline{\partial \dF}|+O(n^{d-2})$ and $\left|{V\choose d}\cap \overline{\partial \dG}\right|=(1-o(1))|\overline{\partial \dG}|$,  where $\dG=\dF^2\cup \dF^3$, 
      
      \item [(ii)] for each set $F\in \dF^3$,  either $f(F)=\mathbf{u}_S$ for some $S\in \dS$, or $f(F)=\frac{1}{2}\mathbf{u}_S+\frac{1}{2}\mathbf{u}_{S'}$ for some distinct sets
    $S, S'\in \dS$, and

    \item[(iii)]   $\sum_{F \in \dF^3}  f(F)\cdot \mathbf{u}_S  \le 1$ for every  set $S\in  \dS$.
\end{enumerate}
\end{lem}
We proceed with the proof of this lemma. Let $n,d$ be positive integers, and let $\dF\subseteq {n\choose d+1}$ with VC$(\dF)\le d$. 
By Lemma~\ref{lem: good subfamily}, there exist a family $\mathcal{G}\subseteq \dF$ and an assignment $c:=c_{\dG}$ of maximum $\dG$-certificates of the sets in $\dG$, satisfying conclusions (i)--(iii) of Lemma~\ref{lem: good subfamily}.  By Lemma~\ref{lem: final refinement}, there exist distinct $i,j\in [n]$ such that $i,j,\dG(i,j),\dG_{d-1}(i),\dG_{d-1}(j)$ satisfy conclusions (i)--(iii) of Lemma~\ref{lem: final refinement} and, without loss of generality, we may assume $i=1$ and $j=2$. 

Set $V:=[n]\setminus [2]$,  $\dF^1:=\dF\setminus \dG$, $\dF^2:=\dG(1,2)\cup\dG_{d-1}(1)\cup \dG_{d-1}(2)$, and $\dF^3:=\dF\setminus (\dF^1\cup \dF^2)$. Then 
$\dG=\dF^2\cup \dF^3$,  $|\dF^1\cup\dF^2|\le 0.1|\overline{\partial \dF}|+O(n^{d-2})$ (by (i) of Lemma~\ref{lem: good subfamily} and (ii) of Lemma~\ref{lem: final refinement}), and  $\left|{V\choose d}\cap \overline{\partial \dG}\right|=(1-o(1))|\overline{\partial \dG}|$ (by (i) of Lemma~\ref{lem: final refinement}).
Hence, we have (i). 

The next step is to define a function from $\dF^3$ to $\mathbf U_{\dS}$, the standard basis of the vector space $\mathbb{R}^{|\dS|}$ indexed by a fixed linear ordering of the sets in $\dS$, where  
$$\dS:=\binom{V}{d-1} \cup \left(\binom{V}{d}\setminus \overline{\partial \dF^3}\right),$$ 
and prove that (ii) and (iii) also hold.

In order to describe the function $f:\dF^3\to \mathbf U_{\dS}$, 
we partition $\dF^3$ into  disjoint families (some of which may be empty). First, we partition $\dF^3$ to disjoint families $\dH$ and $\dK$, with
$$\dH = \{F \in \dF^3 : F \not\subseteq V \} 
\text{ and }
\dK = \{F \in \dF^3 : F \subseteq V \}.$$
Note that, since $F \notin \dG(1,2)\cup\dG_{d-1}(1)\cup \dG_{d-1}(2)$, 
\[\mbox{for each $F\in \dH$, we have
$|c(F)|=d$ and $|c(F)\cap [2]|\le 1$.}\] 
Next, we partition $\dK$ to disjoint families $\dK_d$ and $\dK_{d-1}$, where
$$    \dK_{d} = \{F \in \dK : |c(F)|=d\} 
\text{ and }
     \dK_{d-1} = \{F \in \dK : |c(F)|=d-1\}.$$
We further partition $\dH$ to three disjoint families $\dH_{0,*}, \dH_{1,1}, \dH_{1,2}$, with subscripts corresponding to the pair $\left(|c(F)\cap [2]|,|F\cap [2]|\right)$, as follows:
\begin{itemize}
    \item $\dH_{0,*} =\{F\in \dH: |c(F)\cap [2]|=0\}$;
    \item $\dH_{1,1} =\{F\in \dH: |c(F)\cap [2]|=|F\cap [2]|=1\}    $;
    \item $\dH_{1,2} =\{F\in \dH: |c(F)\cap [2]|=1 \text{ and } |F\cap [2]|=2\}  $.
    \end{itemize}

\begin{dfn}\label{dfn: f}
Define $f:\dF^3\to \mathbf U_{\dS}$  as follows. For $F\in \dF^3\setminus \dK_{d-1}$, let
\[
f(F) =
\begin{cases}
\mathbf{u}_{c(F)}   & \text{if } F \in\dH_{0,*} \cup  \dK_{d};  \\
\mathbf{u}_{F\cap V}    & \text{if } F \in \dH_{1,2}; \\
\frac{1}{2}\mathbf{u}_{F\cap V} + \frac{1}{2}\mathbf{u}_{c(F)\cap V}    & \text{if } F \in \dH_{1,1}. 
\end{cases}
\]
To define $f$ on $\dK_{d-1}$, 
we consider $c^{-1}(T)\cap \dK_{d-1}$ for all $T\in c(\dF^3)$, which may be divided into three different categories according to (i)--(iii) of Lemma~\ref{clm: aT details}:
\begin{itemize}
    \item [(i)] If  $c^{-1}(T)\cap \dK_{d-1}= \{Txy, Tyz, Tzx\}$ for some distinct elements $x,y,z\in [n]\setminus T$, then let $f(Txy)=\mathbf{u}_{Tx}$, $f(Tyz)=\mathbf{u}_{Ty}$, and $f(Tzx)=\mathbf{u}_{Tz}$.

    \item [(ii)] If $c^{-1}(T)\cap \dK_{d-1}= \{Txy, Txz\}$ for some distinct elements $x,y,z\in [n]\setminus T$, then let  $f(Txy)=\mathbf{u}_{Ty}$ and  $f(Txz)=\mathbf{u}_{Tz}$.

    \item [(iii)] If $c^{-1}(T)\cap \dK_{d-1}= \{Txy\}$ for some distinct elements $x,y\in [n]\setminus T$, then  let $f(Txy)$ be a convex combination of $\mathbf{u}_{T}, \mathbf{u}_{Tx}, \mathbf{u}_{Ty}$, with coefficients in $\{0, \frac 12, 1\}$.
\end{itemize}
\end{dfn}

\begin{Claim}\label{clm: mapping size}
    The function $f$ is well defined and, for each set $F\in \dF^3$,  either $f(F)=\mathbf{u}_S$ for some $S\in \dS$, or $f(F)=\frac{1}{2}\mathbf{u}_S+\frac{1}{2}\mathbf{u}_{S'}$ for some distinct
    $S, S'\in \dS$. 
\end{Claim}
\begin{proof} 
  By Lemma~\ref{clm: aT details}, $\dK_{d-1}$ admits a partition into $c^{-1}(T)\cap \dK_{d-1}$ for $T\in c(\dF^3)$, of the form described in (i)--(iii) of Definition~\ref{dfn: f}. By definition of $f$, for each $F\in \dF^3$, $f(F)$ has the form $\mathbf{u}_S$ or $\frac{1}{2}\mathbf{u}_S+\frac{1}{2}\mathbf{u}_{S'}$, where $S,S'$ are subsets of $F$ and both contain $T$. (Note that the convex combination in (iii) has at most two terms.) Hence, to prove this claim, it suffices to check that these $S$ or $S'$ are in ${V\choose d-1}\cup \left({V\choose d}\setminus \overline{\partial \dF^3}\right)$. We have four cases based on Definition~\ref{dfn: f}. Note that for any $F\in \dF^3$, $c(F)\notin \overline{\partial \dF^3}$ (as either $|c(F)|=d-1$ or $c(F)\in \partial \dF^3$).

        If $F \in \dK_{d} \cup \dH_{0,*} $ then $|c(F)|=d$ and $c(F)\subseteq V$; so  $c(F)\in \binom{V}{d}\setminus \overline{\partial \dF^3} $. Hence, $f(F)=\mathbf u_{c(F)}$ is well defined.  
        
         If $F \in \dH_{1,2}$ then $| F\cap [2]|=2$; so $F\cap V\in \binom{V}{d-1}$. Hence $f(F)=\mathbf u_{F\cap V}$ is well defined.     
        
         Now suppose $F \in \dH_{1,1}$. Then $| c(F) \cap [2]|=| F\cap [2]|=1$; so  $ F\cap V\in \binom{V}{d}\setminus \overline{\partial \dF^3}$ and  $c(F)\cap V\in {V\choose d-1}$.  Hence, $f(F)=\frac{1}{2}\mathbf{u}_{F\cap V} + \frac{1}{2}\mathbf{u}_{c(F)\cap V}$ is also well defined.

        Finally, consider $F\in \dK_{d-1}$. Let $T=c(F)$ and $F=Txy$. Note that $Txy\subseteq V$ by definition of $\dK$. By (iii) of  Definition~\ref{dfn: f},   $f(F)$ is a convex combination of $\mathbf u_T, \mathbf u_{Tx}, \mathbf u_{Ty}$ (and possibly $\mathbf u_{Tz}$ with  $z\in V\setminus (T\cup \{x,y\})$) with coefficients in $\{0, \frac 12, 1\}$. Clearly, $T\in {V\choose d-1}$ and  $Tx,Ty,Tz\in {V\choose d}\setminus \overline{\partial \dF^3}$.  Hence, $f(F)$ is well defined. 
 \end{proof}

So (ii) holds. To prove that (iii) also holds, we prove a convenient claim, which, for a given $S\in \dS$, bounds the sum of inner products $f(F)\cdot \mathbf u_S$ over all sets in $\dH_{1,1}$.

\begin{Claim}\label{clm: H1,2}
    For each $S\in \dS $,  $ \sum_{F\in \dH_{1,1}}f(F) \cdot \mathbf{u}_S   \le 1.$ 
\end{Claim}
\begin{proof} 
For each $F\in \dH_{1,1}$, we have $f(F)=\frac 12 \mathbf u_{F\cap V}+\frac 12 \mathbf u_{c(F)\cap V}$ (by definition of $f$), and we have $F\cap V\in \binom{V}{d}\setminus \overline{\partial \dF^3}$ and  $c(F)\cap V\in {V\choose d-1}$ (by definition of $\dH_{1,1}$). 

Suppose  $|S|=d$. If $F\in \dH_{1,1}$ and $f(F) \cdot \mathbf{u}_S \neq 0$, then $F\cap V=S$ and $f(F) \cdot \mathbf{u}_S =1/2$. It follows that $F\in \{S1, S2\}$ and, hence,  $\sum_{F\in \dH_{1,1}}f(F) \cdot \mathbf{u}_S  \le 1/2 +1/2 =1.$ 

Suppose $|S|=d-1$. If $F\in \dH_{1,1}$ and $f(F) \cdot \mathbf{u}_S \neq 0$, then $c(F)\cap V=S$ and $f(F) \cdot \mathbf{u}_S =1/2$. It follows that $c(F)\in \{S1,S2\}$.  Since $|S1|=|S2|=d$, it follows from Lemma~\ref{clm: aT details} that the number of sets in $\dF$ containing $S1$ (respectively, $S2$) is at most one. Hence,  there are at most two sets $F$ in $\dH_{1,1}$ satisfying $f(F) \cdot \mathbf{u}_S \neq 0$.
This implies that $ \sum_{F\in \dH_{1,1}}f(F) \cdot \mathbf{u}_S \le 1/2 +1/2 =1.$ 
\end{proof}

Let $\dS_0:=\bigcup_{F\in \dK_{d-1}} \dS_{F,c}$, where $\dS_{F,c}=\{S\in {F\choose d-1}
\cup {F\choose d}:  c(F)\subseteq S\}$ are the families involved in (iii) of Lemma~\ref{lem: good subfamily}. 
By (iii) of Definition~\ref{dfn: f}, for each $F\in \dK_{d-1}$,    $f(F)$ is a convex combination of the vectors  $\{\mathbf{u}_{S}: S\in \dS_{F,c}\}$. 
The following claim deals with sets in $\dS\setminus \dS_0$.

\begin{Claim}\label{clm: coordinate S}
 For each $S\in \dS\setminus \dS_0$,  we have  $ \sum_{F\in \dF^3}f(F) \cdot \mathbf{u}_S   \le 1$.
\end{Claim}
\begin{proof}
Note that $f(F) \cdot \mathbf{u}_S=0$ for all $F\in \dK_{d-1}$,  since if $F\in \dK_{d-1}$ then  $f(F)$ is a convex combination of vectors in  $\{\mathbf{u}_{S'}: S'\in \dS_{F,c}\}$ and $\dS_{F,c}\subseteq \dS_0$.
Therefore, it suffices to consider only the sets $F\in \dF^3\setminus \dK_{d-1}= \dH\cup \dK_{d}$. 
We distinguish  two cases. 
\medskip

Case 1. $S\in {V\choose d}$. 

First, suppose $S=c(F')$ for some $F'\in \dF^3$. Then, since $|S|=d$, it follows from Lemma~\ref{clm: aT details} that $\{F\in \dF^3: S\subseteq F\}=\{F'\}$. Hence, by definition, $  \sum_{F\in \dF^3}f(F) \cdot \mathbf{u}_S  =  f(F')\cdot \mathbf{u}_S    \le 1$.  

Now suppose $S\neq c(F)$ for all $F\in \dF^3$.
If $F\in \dH_{0,*}\cup \dK_d$ then by definition
$f(F)=\mathbf u_{c(F)}$ and, hence, $f(F)\ne \mathbf{u}_S$ (as $S\ne c(F)$); so $  f(F) \cdot \mathbf{u}_S  =0$.  If $F\in \dH_{1,2}$ then by definition we have $f(F)=\mathbf U_{F\cap V}$ and $|F\cap V|=d-1$; so $F\cap V\ne S$ (as $|S|=d$) and, hence, $  f(F) \cdot \mathbf{u}_S  =0$. 
Therefore, 
    $  \sum_{F\in \dF^3}f(F) \cdot \mathbf{u}_S  =
     \sum_{F\in \dH_{1,1}}f(F)\cdot \mathbf{u}_S$; so  $\sum_{F\in \dF^3}f(F) \cdot \mathbf u_S  \le 1$ by Claim~\ref{clm: H1,2}.    
\medskip

   Case 2.  $S\in \binom{V}{d-1}$. 
   
   Suppose $F\in  \dH\cup \dK_{d}$ with $f(F) \cdot \mathbf{u}_S\neq 0$. Then $F\notin \dH_{0,*}\cup \dK_d$; for, otherwise, we have $c(F)\ne S$ (as $|c(F)|=d$ by definition) and, hence, $f(F)\cdot \mathbf u_S=\mathbf u_{c(F)}\cdot \mathbf u_S=0$,  a contradiction. Hence, $F\in \dH_{1,1}\cup \dH_{1,2}$. Moreover, 
   \begin{itemize} 
      \item [(1)] if $F\in \dH_{1,1}$ then $c(F)\cap V=S$ (by the definition of $f$) and, thus, $F$ must be of the form $Sij$ and $c(F)= Si$, where $i\in [2]$ and $j\in [n]\setminus (S\cup [2])$, and 
      \item [(2)] if $F\in \dH_{1,2}$ then $F\cap V=S$ (by the definition of $f$) and   $F$ must be the set  $S12$ (by the definition of $\dH_{1,2}$).          
   \end{itemize} 

    We claim that $\{F\in \dH_{1,1}: f(F)\cdot \mathbf u_S\ne 0\}=\emptyset$ or $\{F\in \dH_{1,2}: f(F)\cdot \mathbf u_S\ne 0\}=\emptyset$.    
    For, otherwise, by (1) and (2), there exist  $F_1=Sij\in \dH_{1,1}$ (with $c(F_1)=Si$) and $F_2=S12\in \dH_{1,2}$;  but then $F_1\cap F_2=Sij\cap S12=Si=c(Sij)$, a contradiction. 

If $\{F\in \dH_{1,1}: f(F)\cdot \mathbf u_S\ne 0\}=\emptyset$ then we have, by (2),  
    $ \sum_{F\in \dF^3}f(F) \cdot \mathbf{u}_S =f(S12)\cdot \mathbf u_S \le 1$.    
If $\{F\in \dH_{1,2}: f(F)\cdot \mathbf u_S\ne 0\}=\emptyset$ then $\sum_{F\in \dF^3}f(F) \cdot \mathbf{u}_S  =
     \sum_{F\in \dH_{1,1}}f(F) \cdot \mathbf{u}_S$; so $\sum_{F\in \dF^3}f(F) \cdot \mathbf{u}_S \le 1$ by  Claim~\ref{clm: H1,2}.
   \end{proof}

Next we deal with those sets in $\dS_0$. By (iii) of Lemma~\ref{lem: good subfamily}, we see that,  for $F,F'\in \dK_{d-1}$ with $c(F)\ne c(F')$, we have $\dS_{F,c}\cap \dS_{F',c}=\emptyset$. Recall that for each $F\in \dK_{d-1}$, $f(F)$ is a convex combination of the vectors  $\{\mathbf{u}_{S}: S\in \dS_{F,c}\}$. Therefore, 
\begin{center}
if $S\in \dS_{F',c}$ for some $F'\in \dK_{d-1}$ then $f(F)\cdot \mathbf u_S=0$  for every $F\in \dK_{d-1}$ with $c(F)\ne c(F')$.
\end{center}
The next lemma deals with the case when  $S\in \dS_{F,c}$ for those $F\in \dF^3$ with  $|c^{-1}(c(F))\cap \dK_{d-1}|=2$ or $|c^{-1}(c(F))\cap \dK_{d-1}|=3$.  

\begin{Claim}\label{clm: type 2,3}
    Suppose $T\in c(\dF^3)$ and that $x,y,z\in [n]\setminus T$ are distinct, such that $c^{-1}(T)\cap \dK_{d-1}=\{Txy,Tyz,Tzx\}$ or $c^{-1}(T)\cap \dK_{d-1}=\{Txy,Tzx\}$. Then
     $ \sum_{F\in \dF^3}f(F) \cdot \mathbf{u}_S  \le 1$  for $S\in \dS_{Txy,c}\cup \dS_{Txz,c}=\{T,Tx,Ty,Tz\}$.
\end{Claim}

\begin{proof} 
    Note that, for $S\in \{T,Tx,Ty,Tz\}$, if $F\in \dF^3$ such that $f(F)\cdot \mathbf u_S\neq  0$ then, by Definition~\ref{dfn: f}, we have  $S\subseteq F$, which implies that $T\subseteq S\subseteq F$. Therefore, we only need to consider those sets $F\in \dF^3$ with $T\subseteq F$.

    Suppose $c^{-1}(T)\cap \dK_{d-1}=\{Txy,Tyz,Tzx\}$. Then by Lemma~\ref{clm: aT details}, $\{F\in \dF^3: T\subseteq F\}=\{Txy,Tyz,Tzx\}$. Thus,  by (i) of Definition~\ref{dfn: f}, we have  
  $$\sum_{F\in \dF^3}f(F) \cdot \mathbf{u}_S  =\left( f(Txy)+ f(Tyz)+ f(Tzx) \right) \cdot \mathbf{u}_S\\ 
    = (\mathbf{u}_{Tx}+\mathbf{u}_{Ty}+\mathbf{u}_{Tz})\cdot \mathbf u_S 
    \le 1.$$

    Now assume $c^{-1}(T)\cap \dK_{d-1}=\{Txy,Tzx\}$. 
    
    We claim that, for any  $F\in \dF^3\backslash \{Txy,Tzx\}$ such that $T\subseteq F$ and $f(F)\cdot \mathbf u_S\neq  0$ for some $S\in \{T,Tx,Ty,Tz\}$, we have $F\in \{Tx1,Tx2\}\cap \dH_{1,1}$. 
    First, we show that $F\notin \dK_{d-1}$ and $\notin \dH_{0,*}\cup \dK_d$. Suppose $F\in \dK_{d-1}$. Then $c(F)\ne T$ as $F\notin c^{-1}(T)\cap \dK_{d-1}$. Thus by (iii) of Lemma~\ref{lem: good subfamily}, $\dS_{F,c}\cap (\dS_{Txy,c}\cup \dS_{Txz},c)=\emptyset$. Hence, by Definition~\ref{dfn: f}, we have $f(F)\cdot \mathbf u_S= 0$ for all $S\in \{T,Tx,Ty,Tz\}$, a contradiction. 
    Next, suppose $F\in \dH_{0,*}\cup \dK_{d}$. Then $|c(F)|=d$ and $c(F)=S$. Hence, by Lemma~\ref{clm: aT details}, $F$ is the only set in $\dF^3$ containing $c(F)$; this is a contradiction as $c(F)\subseteq  Txy\in \dF^3$ or $c(F)\subseteq Txz\in \dF^3$, and $F\notin \{Txy,Txz\}$. Therefore, $F\in \dH_{1,2}\cup \dH_{1,1}$, which implies $F\cap [2]\neq \emptyset$.   Note that $Txyz\subseteq V$ since $\{Txy,Txz\}\subseteq \dK_{d-1}$. Hence, by (ii) of Lemma~\ref{clm: aT details},  $F\in \{Tx1,Tx2\}\cap \dH_{1,1}$, completing the proof of the claim.   
    
    By the above claim, we see that,  for all $S\in  \{T,Tx,Ty,Tz\}$,
    $$ \sum_{F\in \dF^3}f(F) \cdot \mathbf{u}_S =\left(f(Txy)+f(Txz)+\sum_{F\in \{Tx1,Tx2\}\cap \dH_{1,1} }f(F)  \right)\cdot \mathbf{u}_S  .$$ 
    Recall that $f(Txy)=\mathbf{u}_{Ty}$ and  $f(Txz)=\mathbf{u}_{Tz}$. For $S\in \{Ty,Tz\}$, we see that $S\not\subseteq Tx1$ and $S\not\subseteq Tx2$; so  $ \sum_{F\in \dF^3}f(F) \cdot \mathbf{u}_S  =  (f(Txy)+f(Txz))\cdot \mathbf{u}_{S}  \le  1$. For $S\in \{T ,Tx\}$, we have  $\sum_{F\in \dF^3}f(F) \cdot \mathbf{u}_{S}=\sum_{F \in \dH_{1,1}}f(F) \cdot \mathbf{u}_{S}$; so  by Claim~\ref{clm: H1,2}, $\sum_{F\in \dF^3}f(F) \cdot \mathbf{u}_{S}\le  1$.
   \end{proof}

We now show that we can specify the convex combination for $f(Txy)$ in (iii) of Definition~\ref{dfn: f}  such that $ \sum_{F\in \dF^3}f(F) \cdot \mathbf{u}_S  \le 1$ for all $S\in \{T,Tx,Ty\}$. 

\begin{Claim}\label{clm: type 1}
   Let $T\in c(\dF^3)$ and $x,y\in [n]\setminus T$ such that $c^{-1}(T)\cap \dK_{d-1}=\{Txy\}$. The convex combination of $\mathbf u_T, \mathbf u_{Tx}, \mathbf u_{Ty}$ for $f(Txy)$ in (iii) of Definition~\ref{dfn: f} may be chosen so that $ \sum_{F\in \dF^3}f(F) \cdot \mathbf{u}_S  \le 1$ for all $S\in \dS_{Txy,c}=\{T,Tx,Ty\}$. 
\end{Claim}
\begin{proof}
    For any $F\in \dF^3\backslash \{Txy\}$ such that $f(F)\cdot \mathbf u_S\neq  0$ for some $S\in \{T,Tx,Ty\}$, we have, by Definition~\ref{dfn: f},  $T\subseteq S\subseteq F$.
    The same argument as in the proof of Claim~\ref{clm: type 2,3} shows that  
    $F\in \dH_{1,2}\cup \dH_{1,1}$ and, hence, $F\cap [2]\neq \emptyset$. Since $Txy\subseteq V$ (as $Txy\in \dK_{d-1}$), it follows from   (iii) of Lemma~\ref{clm: aT details} that  $F\in \{Tx1,Tx2,Ty1,Ty2\}\cap \dH_{1,1}$.  
    Hence, for all $S\in  \{T,Tx,Ty\}$,
    $$ \sum_{F\in \dF^3}f(F) \cdot \mathbf{u}_S =\left( f(Txy)+\sum_{F\in \{Tx1,Tx2,Ty1,Ty2\}\cap \dH_{1,1} }f(F)  \right)\cdot \mathbf{u}_S . $$ 

     Recall that $f(H)=\frac{1}{2}\mathbf{u}_{H\cap V} + \frac{1}{2}\mathbf{u}_{c(H)\cap V}  $ for $H\in \dH_{1,1}$. We now specify the convex combination for $f(Txy)$ according to the size of the family $\{H\in \dH_{1,1}: c(H)\cap V=T\}$.

\medskip
Case 1. $|\{H\in \dH_{1,1}: c(H)\cap V=T\}|=0$. 

In this case, we define  
\[f(Txy)=\mathbf{u}_{T}.\] 
For $H\in \dH_{1,1}$, $H\cap V\ne T$ (as $|H\cap V|=d$) and $c(H)\cap V\ne T$ (as $|\{H\in \dH_{1,1}: c(H)\cap V=T\}|=0)$; so $f(H)\cdot \mathbf u_T= (\frac{1}{2}\mathbf{u}_{H\cap V} + \frac{1}{2}\mathbf{u}_{c(H)\cap V}) \cdot \mathbf u_T = 0$.
Therefore,  $\sum_{F\in \dF^3}f(F) \cdot \mathbf{u}_T  = f(Txy) \cdot \mathbf{u}_T = 1$.
For any  $S\in \{Tx,Ty\}$, $f(Txy)\cdot \mathbf u_S=\mathbf u_T\cdot \mathbf u_S=0$; so 
 $\sum_{F\in \dF^3}f(F) \cdot \mathbf{u}_S  = \sum_{F\in \dH_{1,1}}f(F) \cdot \mathbf{u}_S  \le 1$, by Claim~\ref{clm: H1,2}.

\medskip
Case 2. $|\{H\in \dH_{1,1}: c(H)\cap V=T\}|=1$. 

Let $H\in \dH_{1,1}$ satisfying $c(H)\cap V= T$. Then $f(H) \cdot \mathbf{u}_T>0$, which implies that $H=Tai$ and $c(H)=Ti$, for some $i\in [2]$ and $a\in \{x,y\}$. Thus, by the definition of $f$, $f(Tai)=\frac{1}{2}\mathbf{u}_{T}+\frac{1}{2}\mathbf{u}_{Ta}$. Let $b\in \{x,y\}\setminus \{a\}$ and $j\in [2]\setminus \{i\}$. Note that $Tbi\notin \dF^3$ since $Tbi\cap Tai=Ti=c(Tai)$. 
Therefore, $\{Tx1,Tx2,Ty1,Ty2\}\cap \dH_{1,1}=\{Tai,Taj,Tbj\}\cap \dH_{1,1}$. We define
\[f(Txy)=\frac{1}{2}\mathbf{u}_{T}+\frac{1}{2}\mathbf{u}_{Tb}.\] 

Clearly, $f(Txy)\cdot \mathbf u_{Ta}=0$ and, hence,
$\sum_{F\in \dF^3}f(F) \cdot \mathbf{u}_{Ta}  = \sum_{F\in  \dH_{1,1}}f(F) \cdot \mathbf{u}_{Ta}$; so $\sum_{F\in \dF^3}f(F) \cdot \mathbf{u}_{Ta}\le 1$ by Claim~\ref{clm: H1,2}.    Since $c(Tbj)\cap V\ne Tb$ when $Tbj\in \dH_{1,1}$, we have $f(Tbj)\cdot \mathbf u_{Tb}\le 1/2$. Since $Tb\not\subseteq Ta1$ and $Tb\not\subseteq Ta2$, we see that
$\sum_{F\in \dF^3}f(F) \cdot \mathbf{u}_{Tb}  = f(Txy)\cdot \mathbf u_{Tb}+\sum_{F\in \{Tbj\}\cap \dH_{1,1} } f(F)\cdot \mathbf u_{Tb}  \le \frac 12+\frac 12=1$.

 Note that, for $F\in \{Tai,Taj,Tbj\}\cap \dH_{1,1}$, we have $F\cap V\ne T$, as well as $c(F)\cap V\ne T$ unless $F=Tai$.  Hence, $f(F) \cdot \mathbf{u}_T>0$ implies $F=Tai$.  Therefore,  $ \sum_{F\in \dF^3}f(F) \cdot \mathbf{u}_T  = f(Txy)\cdot \mathbf u_T +f(Tai)\cdot \mathbf u_T=\frac 12+\frac 12= 1$.

\medskip

Case 3. $|\{H\in \dH_{1,1}: c(H)\cap V=T\}|\ge 2$. 

Let $H_1,H_2\in \dH_{1,1}$, such that $c(H_i)\cap V = T$ for $i\in [2]$. We may further assume $c(H_i)=Ti$ for $i\in [2]$. 
Moreover, for any set $H\in \dH_{1,1}\setminus \{H_1,H_2\}$, we have $T\not\subseteq H$; since, otherwise, $H=Tuj$ for some $j\in [2]$ and $u\in V\setminus T$ which implies $H\cap H_j=Tj=c(H_j)$, a contradiction. 
Hence, for all $S\in  \{T,Tx,Ty\}$,
    $$ \sum_{F\in \dF^3}f(F) \cdot \mathbf{u}_S =\left(  f(H_1) +f(H_2) +f(Txy) \right)\cdot \mathbf{u}_S . $$

First, consider the case when $H_1\cap H_2\ne T$. Then there exists $a\in \{x,y\}$ such that $H_1\cap H_2=Ta$; so we may assume $H_1=Ta1$ and  $H_2=Ta2$. By the definition of $f$,  $f(H_1)=\frac{1}{2}\mathbf u_{T}+\frac{1}{2}\mathbf u_{Ta}$ and $f(H_2)=\frac{1}{2}\mathbf u_{T}+\frac{1}{2}\mathbf u_{Ta}$.
Let $b\in \{x,y\}\setminus \{a\}$ and define
\[ f(Txy)=\mathbf u_{Tb}, 
\]
Then, for each $S\in  \{T,Tx,Ty\}$, $\left(  f(H_1) +f(H_2) +f(Txy) \right)\cdot \mathbf{u}_S\le 1$. So  $  \sum_{F\in \dF^3}f(F) \cdot \mathbf{u}_S \le 1 . $ 
  
Now suppose $H_1\cap H_2= T$. Then we may assume  $H_1=Ta1$ and $H_2=Tb2$, where  $\{a,b\}= \{x,y\}$.  It follows that $f(H_1)=\frac{1}{2}\mathbf u_{T}+\frac{1}{2}\mathbf u_{Ta}$ and $f(H_2)=\frac{1}{2}\mathbf u_{T}+\frac{1}{2}\mathbf u_{Tb}$. Define 
\[
f(Txy)=\frac{1}{2}\mathbf{u}_{Tx}+\frac{1}{2}\mathbf{u}_{Ty}.
\]
Then, for each $S\in  \{T,Tx,Ty\}$,  $\left(  f(H_1) +f(H_2) +f(Txy) \right)\cdot \mathbf{u}_S\le 1$. So $ \sum_{F\in \dF^3}f(F) \cdot \mathbf{u}_S\le 1$. 
\end{proof}

\begin{proof}[Proof of Lemma~\ref{lem: f G'}] 

By Lemma~\ref{lem: good subfamily}, there exist a family $\mathcal{G}\subseteq \dF$ and an assignment $c:=c_{\dG}$ of maximum $\dG$-certificates of the sets in $\dG$, satisfying conclusions (i)--(iii) of Lemma~\ref{lem: good subfamily}.  By Lemma~\ref{lem: final refinement}, there exist distinct $i,j\in [n]$ such that $i,j,\dG(i,j),\dG_{d-1}(i),\dG_{d-1}(j)$ satisfy conclusions (i) -- (iii) of Lemma~\ref{lem: final refinement}. We may assume $i=1$ and $j=2$. 
Set $V:=[n]\setminus [2]$,  $\dF^1:=\dF\setminus \dG$, $\dF^2:=\dG(1,2)\cup\dG_{d-1}(1)\cup \dG_{d-1}(2)$, and $\dF^3:=\dF\setminus (\dF^1\cup \dF^2)$. Then 
\[\dG=\dF^2\cup \dF^3, ~ |\dF^1\cup\dF^2|\le 0.1|\overline{\partial \dF}|+O(n^{d-2}),  \mbox{ and } \left|{V\choose d}\cap \overline{\partial \dG}\right|=(1-o(1))|\overline{\partial \dG}|. \]
Thus (i) holds. 

Let $\dS={V\choose d-1}\cup \left({V\choose d}\setminus \overline{\partial \dF^3}\right)$ and $\mathbf U_{\dS}=\{\mathbf u_S:S\in \dS\}$ be the standard basis of $\mathbb{R}^{|\dS|}$ indexed by a fixed linear ordering of the sets in $\dS$. 
Let $f$ be defined as in Definition~\ref{dfn: f} and in the proof of Claim~\ref{clm: type 1}. Then by Claim~\ref{clm: mapping size}, $f$ satisfies  (ii). 
Claim~\ref{clm: coordinate S} shows that (iii)  holds for $S\in \dS\setminus \dS_0$, where $\dS_0=\bigcup_{F\in \dK_{d-1}}S_{F,c}$. 
Note that $\dS_0$ can be partitioned into pairwise disjoint families $\dS_{F,c}=\{S\in {F\choose d-1}\cup {F\choose d}: c(F)\subseteq S\}$ with $F\in \dK_{d-1}$. 
 Combining Claims~\ref{clm: type 2,3} and \ref{clm: type 1}, we know that conclusion (iii) also holds for $S\in \dS_0$.
\end{proof}

\section{Conclusion}

In this section, we complete the proof of Theorem~\ref{thm: main} and mention several interesting open problems. First, we modify the function $f$ to obtain an injection $g: \dF^3\to  \mathbf U$, the standard basis of $\mathbb{R}^m$ (for an appropriate choice of $m$). 
  
 \begin{lem}\label{lem: map G' to 1}
Let $n,d$ be positive integers and let $\dF\subseteq \binom{[n]}{d+1}$ be a family with VC$(\dF) \le d$. Then there exist $V\in {[n]\choose n-2}$ and a partition of $\dF$ into families $\dF^1$, $\dF^2$ and $\dF^3$, such that $|\dF^1\cup \dF^2|\le 0.1|\overline{\partial \dF}|+O(n^{d-2})$ and $\left|{V\choose d}\cap \overline{\partial \dG}\right|=(1-o(1))|\overline{\partial \dG}|$, where $\dG=\dF^2\cup \dF^3$, and there exists  an injection $g: \dF^3 \to \mathbf U_{\dS}$, where    $\mathbf U_{\dS}$ is the standard basis of $\mathbb{R}^{|\dS|}$ indexed by a fixed linear ordering of the sets in $\dS=\binom{V}{d-1} \cup \left(\binom{V}{d}\setminus \overline{\partial \dF^3}\right)$. 
\end{lem}

\begin{proof}
    By Lemma~\ref{lem: f G'}, there exist $V\in {[n]\choose n-2}$ and a partition $\dF^1\cup \dF^2\cup \dF^3$ of $\dF$, such that $|\dF^1\cup \dF^2|\le 0.1|\overline{\partial \dF}|+O(n^{d-2})$ and $\left|{V\choose d}\cap \overline{\partial \dG}\right|=(1-o(1))|\overline{\partial \dG}|$, where $\dG=\dF^2\cup \dF^3$, and there exists a function $f:\dF^3\to \mathbb{R}^{|\dS|}$ satisfying (i)--(iii) of Lemma~\ref{lem: f G'}, where $\mathbf U_{\dS}$ is the standard basis of $\mathbb{R}^{|\dS|}$ indexed by a fixed linear ordering of the sets in $\dS=\binom{V}{d-1} \cup \left(\binom{V}{d}\setminus \overline{\partial \dF^3}\right)$.   We now modify $f$ to obtain the desired injection $g$. 
    
     Let $\dF_1^3=\{F\in \dF^3: f(F)=\mathbf u_S \mbox{ for some } S\in \dS\}$, and $\dF_2^3=\{F\in \dF^3: f(F)=\frac 12 \mathbf u_S +\frac 12 \mathbf u_{S'} \mbox{ for some distinct } S,S'\in \dS\}$. Then, by (ii) of  Lemma~\ref{lem: f G'},   $\dF^3_1\cap \dF^3_2=\emptyset$ and $\dF^3=\dF^3_1 \cup \dF^3_2$.  Let $\mathbf{U}_1=\{f(F): F \in \dF^3_1\}$, and let $\mathbf{U}_2=\{\mathbf u_S,\mathbf u_{S'}: \text{ there exists } F\in \dF^3_2 \text{ such that } f(F) = \frac 12 \mathbf u_S+\frac 12\mathbf u_{S'}\}$. Thus, $\mathbf U_i\subseteq \mathbf U_{\dS}$ for $i\in [2]$. 

Since  $\sum_{F \in \dF^3}  f(F)\cdot \mathbf{u}_S  \le 1$ for all $S\in \dS$ (by  (iii) of Lemma~\ref{lem: f G'}), we know that   $\mathbf{U}_1$ and $\mathbf{U}_2$ are disjoint, and each vector in  $\mathbf{U}_2$ appears in $f(F)$ for at most two different sets $F$ in $\dF^3_2$. Therefore, $|\dF^3_2|\le |\mathbf{U}_2|$, since, for each $F\in \dF^3_2$, $f(F)$ involves two standard basis vectors in $\mathbf U_2$.
 Hence, we may let $f':\dF_2^3\to \mathbf U_2$ be an arbitrary injection. 
 
    Now we define $g: \dF^3 \to \mathbf U_{\dS}$, such that $g(F)=f(F)$ for  $F\in \dF^3_1$ and $g(F)=f'(F)$ for $F\in \dF^3_2$. Note that $g(F_1)\ne g(F_2)$ when $F_1,F_2\in \dF^3$ are distinct. Thus, $g$ is an injection and, for each $F \in \dF^3$, $g(F)$ is a standard basis vector of $\mathbb{R}^{|\dS|}$.
\end{proof}

\begin{proof}[Proof of Theorem~\ref{thm: main}]
    Let $\dF\subseteq \binom{[n]}{d+1}$ be a family with VC$(\dF) \le d$. 
    By Lemma~\ref{lem: map G' to 1}, there exist $V\in {[n]\choose n-2}$ and a partition of $\dF$ into families $\dF^1$, $\dF^2$ and $\dF^3$, such that $|\dF^1\cup \dF^2|\le 0.1|\overline{\partial \dF}|+O(n^{d-2})$ and $\left|{V\choose d}\cap \overline{\partial \dG}\right|=(1-o(1))|\overline{\partial \dG}|$ (where $\dG=\dF^2\cup \dF^3$) and there exists an injection $g: \dF^3 \to \mathbf U_{\dS}$, where  $\mathbf U_{\dS}$ is the standard basis of $\mathbb{R}^{|\dS|}$ indexed by a fixed linear ordering of the sets in  $\dS=\binom{V}{d-1} \cup \left(\binom{V}{d}\setminus \overline{\partial \dF^3}\right)$.  Therefore,   
\begin{align*}
    |\dF|&=|\dF^1\cup \dF^2|+|\dF^3|\\
   &\le 0.1|\overline {\partial \dF}|+O(n^{d-2})+|\dS| \\
   &=0.1|\overline {\partial\dF}|+O(n^{d-2})+{n-2\choose d-1}+{n-2\choose d}-|\overline{\partial \dF^3}|\\
   &\le  0.1|\overline {\partial\dG}|+O(n^{d-2})+{n-1\choose d}- |\overline{\partial \dG}| \quad \text{ (since }\overline{\partial \dG}\subseteq \overline{\partial \dF^3})\\
   &={n-1\choose d}+O(n^{d-2})  -0.9|\overline{\partial\dG}|\\
   &\le {n-1\choose d}+O(n^{d-2}),
   \end{align*}
    as desired.
\end{proof}

Note that the above proof actually gives $|\dF|\le {n-1\choose d}+O(n^{d-2})  -0.9|\overline{\partial\dG}|$. Since $\overline{\partial \dF}\subseteq \overline{\partial \dG}$ (as $\dG\subseteq \dF$), we have the following consequence. 
\begin{corollary}\label{cor: partial}
    Let $n,d$ be positive integers and let $\dF\subseteq \binom{[n]}{d+1}$ with  VC$(\dF) \le d$.  Then   $$|\overline{\partial \dF}|\le \frac{10}{9} \left(\binom{n-1}{d}- |\dF|+ O(n^{d-2}) \right).$$

\end{corollary}
 If $\dF$ in Corollary ~\ref{cor: partial} is a  maximum size such family, then $|\dF|\ge {n-1\choose d}+{n-4\choose d-2}$ and, hence, we have $|\overline{\partial \dF}|=O(n^{d-2})$.  Theorem~\ref{thm: main} asymptotically determines the maximum size of a subfamily of ${[n]\choose d+1}$ with VC-dimension at most $d$, as the best lower bound is ${n-1\choose d}+{n-4\choose d-2}$. It appears that in general the exact maximum size of such a uniform set family is difficult to determine. However,   Wang, Xu, and Zhang \cite{25WXZ}  recently determined  the maximum size of a 3-uniform family with VC-dimension 2.
 
 An original motivation for the initial conjectured bound of ${n-1\choose d}$ of Frankl and Pach \cite{84FP} (also mentioned by Erd\H{o}s \cite{84Er})  was to generalize the famous Erd\H{o}s-Ko-Rado theorem. However, it turned to be not the case.  In an attempt to remedy this situation, Chao {\em et al} \cite{CXYZ} considered a uniform version of the Frankl-Pach conjecture, which would provide an interesting generalization of the Erd\H{o}s-Ko-Rado theorem. 

 \begin{conjecture}[Chao, Xu, Yip, and Zhang, 2024]\label{conj: uniform}
 Let $n\ge 2(d+1)$ and let $0\le s\le d$. Assume $\dF\subseteq {[n]\choose d+1}$ such that for every $F\in \dF$, there exists an $\dF$-certificate of $F$ of order $s$.
 Then $|\dF|\le {n-1\choose d}$. 
 \end{conjecture}

In \cite{CXYZ}, Conjecture~\ref{conj: uniform} is verified for $s\in \{0,d\}$ and $s=1$ and $n$ large (compared to $d$).  

For the maximum size of a non-uniform set family with bounded VC-dimension, a foundational result, known as the the Shatter Function Theorem, was discovered independently in the early 1970s by Perles–Shelah  \cite{72Sh}, N. Sauer \cite{72Sa}, and Vapnik–Chervonenkis \cite{71VC}. It states that if the VC-dimension of $\dF\subseteq 2^{[n]}$ is $d$, then $|\dF|\le \sum_{i=0}^{d}\binom{n}{i}$. This bound is tight and admits several extremal constructions. A quantitative version of this result was given by Pajor \cite{85Pa}. 

A set family $\dF$ is a {\em Sperner family} if for any distinct members $A, B\in \dF$, $A\not\subseteq B$ and $B\not\subseteq A$. (The set families considered in Theorem~\ref{thm: main} are Sperner families as they are subfamilies of ${[n]\choose d+1}$.)   It is a long standing open problem in extremal set theory to determine the maximum size of a Sperner family with boudned VC-dimension. The following conjecture was made by Frankl \cite{89Fr} in 1989. 

\begin{conjecture} [Frankl, 1989] \label{con: sperner}
    Let $n,d$ be positive integers, and let $\dF\subseteq 2^{[n]}$ be a Sperner family. If $n\ge 2d$ and the VC-dimension of $\dF$ is at most $d$, then $|\dF|\le {n\choose d}$  
\end{conjecture}

Frankl \cite{89Fr} proved a general bound on Sperner families with bounded VC-dimension, which implies that Conjecture~\ref{con: sperner} holds for $k=1$ and $k=2$. Anstee and Sali \cite{97AS} proved Conjecture~\ref{con: sperner} for $k=3$. 

A Sperner family $\dF\subseteq 2^{[n]}$ may be viewed as an antichain in the lattice $2^{[n]}$. It is natural to  consider families $\dF\subseteq 2^{[n]}$ which has bounded VC-dimension and contains no chain of certain length. Indeed, Frankl~\cite{89Fr} made the following conjecture. 

\begin{conjecture}[Frankl, 1989]
    Let $n,d,\ell$ be positive integers such that $n+\ell\ge 2(d+1)$, and let $\dF\subseteq 2^{[n]}$ such that VC$(\dF)\le d$ and $\dF$ contains no chain of length $\ell+1$. Then $|\dF|\le \sum_{i=d-\ell}^d{n\choose i}$.

\end{conjecture}

\bibliographystyle{unsrt}

\end{document}